\documentclass [12pt]{article}
\usepackage[utf8]{inputenc}
\usepackage{amsmath}
\usepackage[T2A]{fontenc}
\usepackage{amsfonts, amssymb}
\usepackage{graphicx}
\usepackage[left=2.5cm,right=2.2cm,top=2.5cm,bottom=3.5cm]{geometry}
\usepackage{comment}

\newenvironment{proof}[1][] {\noindent {\bf Proof#1:} }{\hspace*{\fill}$\square$\medskip\par}

\newtheorem{thm}{Theorem}
\newtheorem{lem}[thm]{Lemma}
\newtheorem{prop}[thm]{Proposition}

\def\R{{\mathbb R}}
\def\E{{\mathbb E\,}}
\def\P{{\mathbb P}}
\def\N{{\mathbb N}}

\def\be#1\ee{\begin{equation}#1\end{equation}}

\def\ed#1{ {\mathbf 1}_{ \{#1  \}}}             % indicator

\def\ow{\overline{w}}
\def\oww{\overline{W}}
\def\orr{\overline{W}^{(r)}}
\def\orrt{\overline{W}^{(r,T)}}
\def\sol{\chi}

\begin{document}
	
	\author{M.A.\,Lifshits\footnote{Saint Petersburg State University, 199034, Saint Petersburg, University Emb., 7/9. {\tt mikhail@lifshits.org}}, S.E.\,Nikitin\footnote{Saint Petersburg State University, 199034, Saint Petersburg, University Emb., 7/9. {\tt nikitin97156@mail.ru}}
	}
	\title{Energy saving approximation of Wiener process \\ under unilateral constraints\footnote{The work supported by RSF grant 21-11-00047.}}
	\date{\today}
	\maketitle

	\begin{abstract}
	We consider the energy saving approximation  of a Wiener process under unilateral constraints. 
	We show that, almost surely,  on  large time intervals the minimal energy necessary for the approximation logarithmically depends on the interval's length. We also construct an adaptive approximation strategy that is optimal in a class of diffusion strategies and also provides the logarithmic order of energy consumption. 	
	\end{abstract}
	
\section{Problem setting and main result}
Let $AC[0,T]$ denote the space of absolutely continuous functions on the interval $[0,T]$.
For $h\in AC[0,T]$ let us call kinetic energy
\[
   	|h|_T^2 := \int_0^T h^\prime(t)^2 dt. 
\]
In the works \cite{BL,KL2,IKL,KL1,LS,LSiu,Scher} the approximation of a random process sample path 
by the function $h$ of smallest energy was considered under various constraints on closeness between $h$ and the sample path. In particular, in \cite{LS} the energy saving approximation was studied for a Wiener process $W$ under  {\it bilateral} uniform constraints.

For $T>0,r>0$ let us define the set of admissible approximations as
\[
    	M^\pm_{T, r} := \left\{h \in AC[0, T] \, \big| \, \forall t \in [0, T] : 
	  W(t) -r \le h(t)  \le  W(t)+ r ; h(0) = 0 \right\}
\]
and let
\[
   	I^\pm_W(T,r) := \inf \left\{|h|_T^2 \: \big| \: h \in M^\pm_{T,r} \right\}.
\]
It was proved in \cite{LS} that for every fixed $r>0$  it is true that
\[
   	\frac{I^\pm_W(T,r)}{T} 
	  \stackrel{\text{a.s.}}{\longrightarrow} {\mathcal C}^2\, r^{-2},
        \qquad \textrm{as } T\to \infty,
\]
where ${\mathcal C}\approx 0,63$ is some absolute constant (the exact value of ${\mathcal C}$ is unknown), i.e. 
the optimal approximation energy grows linearly in time.
	
In this work, we are interested in the behavior of a similar quantity under {\it unilateral} constraints, i.e. the set of
admissible approximations is
\[
	M_{T, r} := \left\{h \in AC[0, T] \, \big| \, \forall t \in [0, T] : 
	h(t)  \ge  W(t) -r ; h(0) = 0 \right\},
\]
and we are now interested in the behavior of
\[
	I_W(T,r) := \inf \left\{|h|_T^2 \: \big| \: h \in M_{T,r} \right\}.
\]
It is technically more convenient to translate the initial value of the approximating function to the point $r$, 
so that this function runs above the trajectory of the approximated process $W$. Let
\[
	M_{T, r}^\prime :=  \left\{h \in AC[0, T] \, \big| \, \forall t \in [0, T] : 
	h(t) \ge W(t); h(0) = r \right\}.
\]
Since the sets of functions $M_{T, r}$ and $M_{T, r}^\prime$ differ by a constant shift, it is easy to see that
\[
	I_W(T,r) = \inf \left\{|h|_T^2 \:\big | \: h \in M_{T,r}^\prime \right\}.
\]
	
Our main result asserts that, when $T$ grows, the quantity $I_W(T,r)$ grows merely logarithmically.
	
\begin{thm} \label{t:main} 
For every fixed $r>0$ it is true that
\begin{equation} \label{asymp_f}
			\frac{I_W(T,r)}{\log T} \stackrel{\text{a.s.}}{\longrightarrow} \frac12, 
            \qquad \textrm{as } T\to \infty. 
\end{equation}  
	\end{thm}
	
In Section \ref{s:2} we establish a connection between the unilateral energy saving approximation of
arbitrary continuous function with its minimal concave majorant. 
In Section \ref{s:3} we establish the necessary properties of Wiener process minimal concave majorant using the results of Groeneboom \cite{concMaj}. Section \ref{s:4} contains the proof of Theorem~\ref{t:main}. 
	
In Sections \ref{s:5}--\ref{s:6} we consider a class of adaptive Markovian (diffusion) approximation strategies 
based on the current and past values of $W$. We prove that in this class the optimal strategy is defined by the formula
\[
	h'(t) =  \frac{1}{h(t)-W(t)} \,. 
\]
For this strategy the energy consumption also has the logarithmic order but is two times larger than that for the optimal non-adaptive strategy using the information about the whole trajectory of $W$.
Namely,  
\[
	\frac{ |h|_2^2} {\log T} \stackrel{\text{a.s.}}{\longrightarrow} 1,
    \qquad \textrm{as } T\to \infty.
\]

\section{Concave majorants as efficient approximations}
\label{s:2}
	
It turns out that the optimal energy saving approximation under unilateral constraints may be described in terms of
the minimal concave majorant (MCM) of the approximated function. 
Let $w:[0,T]\mapsto \R$  be a continuous function. Then the corresponding MCM $\ow$ is the minimal concave function satisfying
conditions
\[
	\ow (t) \ge w(t), \qquad 0\le t\le T.
\]
		
\begin{prop} \label{maj_form} 
Let $r>w(0)$. Then the problem $|h|_T^2\to \min$ under the constraints $h(0)=r$ and
\[
		h(t) \ge w(t), \qquad 0\le t\le T,
\]
has a unique solution $\chi_*$ of the following form. 
		
(a) If $r\ge \max_{0\le t\le T} w(t)$, then $\chi_*(t)\equiv r$.
		
(b) If $r<\max_{0\le t\le T} w(t)$, then $\chi_*$ is defined differently on three intervals.
On the initial interval, $\chi_*$ is an affine function whose graph contains the point $(0,r)$ and is a tangent to the graph
of $\ow$. 
Then $\chi_*$ coincides with $\ow$ until the first moment when the maximum of $w$ is attained. 
Finally, after that moment, $\chi_*$ is a constant.
\end{prop}
\medskip

The optimal energy saving majorant $\chi_*$ is shown in Figure~\ref{fig:eef}.

\begin{figure*}[ht]
	\centering
    \includegraphics[scale=0.3]{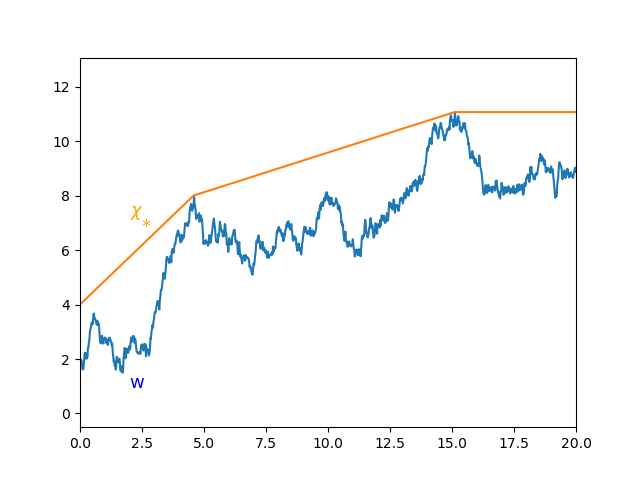}
	\caption{The optimal energy saving majorant. }
	\label{fig:eef}
\end{figure*}
	\medskip
	
\begin{proof}[ of the proposition]  The problem's solution exists since for every $M>0$ the set of functions 
\[
		\{ h\in AC[0,T]: h(0)=r, h\ge w, |h|_T\le M\}
\]   
is compact in the space of continuous functions equipped with the topology of uniform convergence and the functional
$|\cdot|_T^2$ is lower semi-continuous in this topology.
		
The uniqueness of the solution follows from the fact that the set of functions satisfying problem's assumptions
is convex and the functional $|\cdot|_T^2$ is strictly convex on this set.

Let us describe the solution.
		
Since  case (a) is trivial, we consider case (b). 		   
		
Let $\chi(\cdot)$ be the solution of our problem. We show first that $\chi$ is a convex non-decreasing function. 
Consider the function
\[
		\chi_1(t):= r +\int_0^t g(s) ds, \qquad 0\le t \le T,
\]
where $g(\cdot)$ is the non-increasing monotone rearrangement of the function $\max\{\chi'(\cdot),0\}$.
Then $\chi_1$ is a concave non-decreasing function, $\chi_{1}(0)=r$ and $\chi_{1}(t) \ge \chi(t) \ge w(t)$ for all $t\in[0,T]$. Therefore, $\chi_{1}$ satisfies the problem's constraints.
On the other hand,
\[
		|\chi_{1}|_T^2=\int_0^T g(s)^2 ds = \int_0^T  \max\{\chi'(\cdot),0\}^2 ds 
		\le  |\chi|_T^2.
\]
Due to the problem solution uniqueness we obtain $\chi_1=\chi$. This equality proves that $\chi$ is concave and non-decreasing.
		
Since $\chi_*$ is the smallest concave non-decreasing function satisfying problem's constraints, we have
\[
		\chi(t)\ge \chi_*(t), \qquad 0\le t\le T.
\]
		
Furthermore, let us prove that $\chi(T)=\chi_*(T)=\max_{0\le t\le T} w(t)$.
Indeed, in case (b) the function 
\[
		\chi_2(t):= \min\{ \chi(t), \chi_*(T) \}, \qquad 0\le t \le T,
\]
satisfies both problem constraints and $|\chi_{2}|_T^2 \le |\chi|_T^2$; by the uniqueness of the solution,
we obtain $\chi=\chi_2$. In particular, $\chi(T)=\chi_*(T)$.
		
Finally, assume that the strict inequality $\chi(t_0)>\chi_*(t_0)$ holds for some $t_0\in [0,T]$. 
Then, since the function $\chi_*$ is concave and non-decreasing, there exists a non-decreasing affine function
$\ell(\cdot)$ such that
\[
		\chi_*(t)\le \ell(t), \qquad 0\le t \le T,
\]
but $\ell(t_0)< \chi(t_0)$. However, at the endpoints of the interval $[0,T]$ the opposite inequality is true,
because
\[
		\chi(0) = r =\chi_*(0) \le \ell(0), \qquad   \chi(T)  =\chi_*(T) \le \ell(0).
\]
Therefore, there exists a non-degenerated interval $[t_1,t_2]\subset[0,T]$ such that
$t_0\in[t_1,t_2]$, $\ell(t_1)=\chi(t_1)$,  $\ell(t_2)=\chi(t_2)$.

Since $\ell'(\cdot)$ is a constant, while $\chi'(\cdot)$ is not a constant on $[t_1,t_2]$,
it follows from H\"older inequality that
\begin{eqnarray*}
			&& (t_2-t_1) \int_{t_1}^{t_2} \chi'(t)^2 dt > \left(\int_{t_1}^{t_2} \chi'(t)dt\right)^2
			= (\chi(t_2)-\chi(t_1))^2
\\
			&=& (\ell(t_2)-\ell(t_1))^2 =  \left(\int_{t_1}^{t_2} \ell'(t)dt\right)^2
			=  (t_2-t_1) \int_{t_1}^{t_2} \ell'(t)^2 dt. 
\end{eqnarray*}
We obtain
\[
		\int_{t_1}^{t_2} \chi'(t)^2 dt >  \int_{t_1}^{t_2} \ell'(t)^2 dt.
\]
It follows that the function
\[
		\chi_3(t):= \min\{ \chi(t), \ell(t) \}, \qquad 0\le t \le T,
\]
satisfies the problem's constraints and  $|\chi_{3}|_T^2 <  |\chi|_T^2$ but this is impossible
by the definition of $\chi$. Therefore, the assumption $\chi(t_0)>\chi_*(t_0)$ brought us to
a contradiction.			 
\end{proof}
	
\section{Minimal concave majorant of a Wiener process}
\label{s:3}	
We recall some notation and results from the article \cite{concMaj} that will be used in the sequel.
Denote	 
\[
	\tau(a) := \sup \left\{t>0 \, \big| \, W(t)-t/a=\sup_{u>0}\left(W(u)-u/a\right)\right\}.
\]
The function $a\mapsto \tau(a)$ is non-decreasing.
	
According to \cite[Corollary~2.1]{concMaj} for every $a>0$ the random variable $\frac{\tau(a)}{a^2}$ has 
the distribution density 
\[
	q(t) = 2 \, \E\left(\frac{X}{\sqrt{t}} - 1\right )_+, \qquad t>0,
\]
where $ x_+ := x \ed{x>0} $, $X$ is a standard normal random variable.
	
Next, let $\oww$ be the global MCM for the Wiener process $W(t),t\ge 0$. 
Define a random process $L$  as
\[
	L(a,b) := \int_{\tau(a)}^{\tau(b)} \oww'(t)^2 dt.
\]
	
Our study is essentially based on the following result due to Groeneboom.
	
\begin{lem} \cite[Theorem~3.1]{concMaj}.
   For every $a_0>0$ the process $X(t) := L(e^{a_0},e^{a_0+t}), t\ge 0$, 
   is a pure jump process with independent stationary increments and $\E X(t)=t$.     
\end{lem}
	
Moreover, there is an explicit description of the L\'evy measure of $X$ in \cite{concMaj} but we do not need it here.
We are only interested in the Kolmogorov's strong law of large numbers for $X$ which asserts that
\[
	\frac{X(t)}{t}   \stackrel{\text{a.s.}}{\longrightarrow} 1,
    \qquad \textrm{as } t\to\infty.
\]

 Using the definition of $X$, letting $a_0=0$, and making the variable change $V=e^t$, we may reformulate this result as
\begin{equation} \label{slln}
		\frac{L(1,V)}{\log V}   \stackrel{\text{a.s.}}{\longrightarrow} 1,
		\qquad \textrm{as } V\to\infty.
\end{equation}
	
\begin{lem} \label{tau_est}
  For every $ \delta \in \left(0, \frac 12\right)$ with probability $1$ for all sufficiently large $T$ it is true that
  \begin{equation} 
			\tau\left(T^{\frac 12 + \delta}\right) > T > \tau\left(T^{\frac 12 - \delta}\right).
  \end{equation}
\end{lem}
	
 \begin{proof}
The lower bound is based upon the inequalities
\begin{eqnarray*}
			\P\left( \tau\left(T^{\frac 12 - \delta}\right) \ge \frac T2 \right) 
			&=& 
			\P\left( \frac{\tau \left( T^{\frac 12 - \delta}\right)}{T^{1-2\delta}} \ge \frac{T^{2\delta}}2 \right)
\\
			&=& \int_{T^{2\delta}/2}^\infty q(t)\, dt
\\
			&=& \int_{T^{2\delta}/2}^\infty 2\,\E\left(\frac{X}{\sqrt{t}} - 1\right )_+ dt
\\
			&=& C_1 \int_{T^{2\delta}/2}^\infty \int_{\sqrt t}^\infty 
			\left(\frac x {\sqrt t} - 1\right )e^{-x^2/2} dx dt
\\
			&\le& C_2 \int_{T^{2\delta}/2}^\infty e^{-t/2} t^{-1/2} dt
			\le C_2 \, e^{-T^{2\delta}/4},
\end{eqnarray*}
where $C_1,C_2$ are some positive absolute constants.
		
Let $ T_n := n $, then the events 
$ D_n := \left\{\tau\left(T_n^{\frac 12 - \delta}\right) \ge \frac {T_n}2 \right \} $ satisfy
\[
		\sum\limits_{n=1}^\infty \P(D_n) < \infty.
\]
Therefore, by Borel--Cantelli lemma, with probability $1$ for all sufficiently large $n$ the event $D_n$ does not hold, i.e. with probability $1$ for all sufficiently large $n$  we have
\begin{equation} \label{upest}
			\tau\left(T_n^{\frac 12 - \delta}\right) < \frac {T_n}2.
\end{equation}
Let $n\ge 2$ be such that \eqref{upest} holds for $T_n$ and let $T \in [T_{n-1}, T_{n}]$. 
Since $\tau(\cdot)$ is non-decreasing, we have 
\[
		\tau\left(T^{\frac 12 - \delta}\right) \le \tau\left(T_{n}^{\frac 12 - \delta}\right) 
		< \frac{T_n}2  = \frac{n}2 \le  n-1 = T_{n-1} \le T.
\]
This provides us with a required lower bound for sufficiently large $T$.
		
In the same way, for the upper bound we have
\begin{eqnarray*}
			\P\left( \tau\left(T^{\frac 12 + \delta}\right) \le 2T \right) 
			&=& 
			\int_0^{2T^{-2\delta}} 2\, \E\left(\frac{X}{\sqrt{t}} - 1\right )_+ dt
\\
			&\le& C_3 \int^{2T^{-2\delta}}_0 e^{-t/2} t^{-1/2} dt
			\le C_4 T^{-\delta},
\end{eqnarray*}
where $C_3,C_4$ are some positive absolute constants.
		
Consider the sequence $ T'_n := 2^n$. For the events 
$ D'_n := \left\{\tau\left(T_n^{\frac 12 + \delta}\right) \le  2 T_n \right \}$ we have 
\[
		\sum\limits_{n=1}^\infty \P(D'_n) < \infty.
\]
Therefore, by Borel--Cantelli lemma with probability $1$ for all sufficiently large $n$ the event $D'_n$ does not hold,
i.e. with probability $1$ for all sufficiently large $n$  it is true that
\begin{equation} \label{lowest}
			\tau\left(T_n^{\frac 12 + \delta}\right) > 2\,T_n .
\end{equation}
		
Let \eqref{lowest} be satisfied for some $T_n$ and let $ T \in [T_{n}, T_{n+1}] $. 
Since $\tau(\cdot)$ is non-decreasing, we have
\[
		\tau\left(T^{\frac 12 + \delta}\right) 
		\ge \tau\left(T_n^{\frac 12 + \delta}\right)  
		> 2 \, T_n = T_{n+1}  \ge T.
\]
This provides us with a required upper bound for all sufficiently large $T$.
		
\end{proof}

The next theorem describes the asymptotic behavior of the energy of the minimal concave majorant for a Wiener process. 
For some $r>0$ let $\orr$ denote MCM of a Wiener process $W$ on the whole real line, starting from the height $r$.
Then the majorant  $\orr$  is an affine function on some initial interval $[0,\theta(r)]$, its graph contains  the point
$(0,r)$ and is a tangent to the graph of $\oww$, while on $[\theta(r), \infty)$ the functions $\orr$ and $\oww$ coinсide. 
	
\begin{thm} \label{asymp_f_th}
  Let $\orr$ be the global minimal concave majorant of $W$ starting from a height $r$. 
  Then, for every fixed $r$ it is true that 
  \begin{equation} %%\label{asymp_f}
			\frac{|\orr|_T^2}{\log T} \stackrel{\text{a.s.}}{\longrightarrow} \frac 12,
   \qquad  \textrm{as } T\to\infty. 
  \end{equation}  
\end{thm}
	
\begin{proof}
Compare the quantities
\[
		|\orr|_T^2 = (\orr)'(0)^2 \theta(r) + \int_{\theta(r)}^T \oww'(t)^2\, dt, 
		\qquad T\ge \theta(r),
\]
and
\[
		L(1,T^{1/2\pm\delta}) = \int_{\tau(1)}^{\tau(T^{1/2\pm\delta})} \oww'(t)^2\, dt.
\]
They differ by a term (independent of  $T$) corresponding to the initial segment of	$\orr$ 
and by the lower and upper integration limits; notice that the lower integration limits do not depend on
$T$ in both cases.
		
By using Lemma~\ref{tau_est} for comparing the upper integration limits, we obtain
\begin{eqnarray*}
			\liminf\limits_{T  \to \infty} \frac{|\orr|_{T }^2}{\log {T }} 
			&\ge&   \liminf\limits_{T \to \infty} 
			\frac {L \left(1, T^{1/2 - \delta}\right)}{\log {T }};
\\             
			\limsup\limits_{T  \to \infty} \frac{|\orr|_{T }^2}{\log {T }} 
			&\le&   \liminf\limits_{T \to \infty} 
			\frac {L \left(1, T^{1/2 + \delta}\right)}{\log {T }}.
\end{eqnarray*}

  Taking into account the law of large numbers \eqref{slln} we have
\[
		\frac12 - \delta  \le
		\liminf\limits_{T  \to \infty} \frac{|\orr|_{T }^2}{\log {T }} 
		\le
		\limsup\limits_{T  \to \infty} \frac{|\orr|_{T }^2}{\log {T }} 
		\le  \frac12 +\delta.
\]
Letting $\delta\searrow 0$ yields the required result.
	\end{proof}
	
\section{Proof of Theorem~\ref{t:main}} 
\label{s:4} \quad
	
{\bf Upper bound.} The restriction of the global MCM starting from the height $r$ onto the interval [0,T]
belongs to the set of admissible functions: $\orr\in M_{T, r}^\prime$. We derive from Theorem~\ref{asymp_f_th} that
\[
	\limsup\limits_{T  \to \infty} \frac{I_W(T,r)}{\log {T }} 
	\le
	\limsup\limits_{T  \to \infty} \frac{|\orr|_{T }^2}{\log {T }}
	\le \frac12 \qquad \textrm{a.s.}
\]
	
{\bf Lower bound.}
For $r>0,T>0$ let $\orrt$ denote the {\it local} MCM of the Wiener process $W(t), t\in [0,T]$, starting from the height $r$.
Let $\sol$ be the unique solution of the problem we are interested in,
$|h|_T^2\to \min, h\in M'_{T,r}$. Recall that its structure is described in Proposition~\ref{maj_form}. 
Since for large $T$ it is true that $\max_{0\le s\le T} W(s)>r$,  for such $T$ the
assumption of case (b) of that proposition is verified. In particular, it follows that
\[
	\sol(t)=\orrt(t),  \qquad 0 \le t \le t_{\max},
\]
where
\[
	t_{\max}=t_{\max}(T) := \min\{t: W(t)= \max_{0\le s\le T} W(s)\}.
\]
Notice that the function $\tau(\cdot)$ can not take values from the interval $(t_{\max},T)$. Therefore, if for some $a$ it is
true that $\tau(a)<T$, then it is also true that $\tau(a)\le t_{\max}$. In this case we have
\[
	\sol(t)=\orrt(t)=\orr(t),  \qquad 0 \le t \le \tau(a).
\]
It follows that
\[
	I_W(T,r)=|\sol|_2^2 \ge \int_0^{\tau(a)} \sol'(t)^2 \, dt
	=|\orr|_{\tau(a)}^2.
\]
Let us fix $\delta\in (0,1/2)$. Let $a=a(T):= T^{1/2-\delta}$. 
Then by Lemma~\ref{tau_est} we have 
\[ 
	T^{\frac{1-2\delta}{1+2\delta}} < \tau(a) < T
\]      
a.s. for all sufficiently large $T$.
Furthermore, it follows from Theorem~\ref{asymp_f_th} that, as $T\to \infty$,
\[
	|\orr|_{\tau(a)}^2 \ge \frac{\log\tau(a)}{2}  \, (1+o(1))
	\ge \frac{1-2\delta}{2(1+2\delta)} \, \log T \, (1+o(1))
	\qquad \textrm{a.s.}
\]
 By combining these estimates, we obtain
 \[
	I_W(T,r) \ge \frac{1-2\delta}{2(1+2\delta)} \, \log T \, (1+o(1))
	\qquad \textrm{a.s.}
\]
Finally, by letting $\delta\searrow 0$, we arrive at
\[
     	I_W(T,r) \ge \frac{1}{2} \, \log T \, (1+o(1))
     	\qquad \textrm{a.s.,}
\]
as required.
	%%%%%%%%%%%%%%%%%%%%%%%%%%%%%%%%%%%%%%%%%%%%%%%%%%%%%%%%%%%%%%%%%%%%%%%%%%%%%%%%%
	
\section{Adaptive Markovian approximation} 
\label{s:5}
	
In practice, it is often necessary to arrange an approximation (a pursuit) in real time (adaptively), when the trajectory of the approximated process is known not on the entire time interval but only before the current time instant. In view of the Markov property of Wiener process, a reasonable strategy is to define the speed of a pursuit $h$ as a function of current positions of the processes $h$ and $W$, without taking past trajectories into account, i.e. let
\begin{equation} \label{adapt_eq1}
		h'(t):= b(h(t),W(t),t).
\end{equation}
On the qualitative level the function $b(x,w,t)$ must tend to infinity, as $x-w\searrow 0$, i.e. when the approximating process approaches the dangerous boundary it accelerates its movement trying to escape from a dangerous position. One has to optimize the function $b$ trying to reach the smallest average energy consumption. It is possible to reach the same logarithmic in time order of energy consumption as in the case
of non-adaptive approximation but with somewhat larger coefficient. The difference of the coefficients represents the price we pay for 
not knowing the future of the process we try to approximate.
	
It is interesting to compare \eqref{adapt_eq1} with the form of the optimal adaptive strategy in the case of bilateral constraints
\cite{LS} where
\begin{equation} \label{adapt_eq2}
		h'(t) = b(h(t)-W(t)).
\end{equation}
The latter strategy is more simple because the speed is governed only by the distance between the approximated and the approximating
processes and does not depend on time.

Let us make a time and space change
\begin{eqnarray*}
		U(\tau) &:=& e^{-\tau/2}  W(e^\tau), 
\\   
		z(\tau) &:=& e^{-\tau/2}  h(e^\tau).
\end{eqnarray*}
Recall that $U(\cdot)$ is an Ornstein--Uhlenbeck process and therefore satisfies the equation
\begin{equation} \label{OU_eq}
		dU = -\frac {U \, d\tau}{2} + d \widetilde{W},
\end{equation}
where $\widetilde{W}$ is a Wiener process. 
We have the following expression for the derivative of $z$ 
\begin{equation}   \label{zprime_eq}
		z^\prime(\tau) = -\frac 12 z(\tau) + e^{\tau/2} h^\prime(e^\tau),   
\end{equation}
which yields
\begin{equation}   \label{hprime_eq}
		h^\prime(e^\tau) =  e^{-\tau/2} \left(z^\prime(\tau) +\frac{z(\tau)}{2}\right).  
\end{equation}
Let us consider the distance between the approximated and approximating processes
\begin{eqnarray}  \label{def_Z_eq}
		Z(\tau) &:=& z(\tau) - U(\tau).
\end{eqnarray}
	
We will study time-homogeneous diffusion strategies
\begin{equation}   \label{diffusion_eq}
		dZ =  b(Z) d \tau - d \widetilde{W}.
\end{equation}
From equations \eqref{OU_eq} and \eqref{diffusion_eq} it follows that this is equivalent to
\[
	z'(\tau) + \frac{U(\tau)}{2}  = b(Z(\tau)),
\]
which also implies
\begin{equation}   \label{zprime2_eq}
		z'(\tau) + \frac{z(\tau)}{2} = b(Z(\tau)) 	+ \frac{Z(\tau)}{2}.
\end{equation}
	
Before proceeding to optimization, let us see how the diffusion strategies act in the  initial framework. 
By \eqref{hprime_eq} and \eqref{zprime2_eq} we have
\begin{eqnarray}
		h'(e^\tau) &=& e^{-\tau/2} \left(b(Z(\tau))+ \frac{Z(\tau)}{2}\right)
		:=  e^{-\tau/2}\ \widetilde{b}(Z(\tau))
\\
		&=&   e^{-\tau/2} \ \widetilde{b}\left( e^{-\tau/2} \left(h(e^\tau)-W(e^\tau)\right) \right),
\end{eqnarray}
where $ \widetilde{b}(x):= b(x)+ \tfrac{x}{2}$.
In other words, the form of the strategy is
\begin{equation}\label{hstrategy_eq}
		h'(t) = \frac{1}{\sqrt{t}}\  \widetilde{b}\left( \frac{1}{\sqrt{t}} 
		\left(h(t)-W(t)\right)\right). 
\end{equation}
We see that this class of strategies is space-homogeneous but, in general, not time-homogeneous.
	
Now we proceed to the optimization of the shift coefficient $b(\cdot)$ determining the pursuit strategy.
Let us use some basic facts about one-dimensional time-homogeneous diffusion, cf. \cite[Ch.IV.11]{Bor} and \cite[Ch.2]{BS}. 
Let
\begin{eqnarray} \label{B}
		B(x) &:=& 2 \int^x b(u) du, 
\\ \label{p0}
		p_0(x) &:=& e^{B(x)}. 
\end{eqnarray}
Assume that condition
\begin{equation}\label{noexit}
		\int_0 \frac{dx}{p_0(x)}  = \infty
\end{equation}
is verified. Then, in Feller classification, the point	$0$ is the entrance-boundary and not an exit-boundary for diffusion \eqref{diffusion_eq}. 
This means that the diffusion $Z$ remains forever in $[0,\infty)$. Moreover, the function
\begin{equation} \label{p}
		p(x) := Q^{-1} p_0(x),
\end{equation}
where $Q = \int_0^\infty p_0(x) dx$, is the density of the unique stationary distribution for $Z$.
For the energy, by using \eqref{hprime_eq} and \eqref{zprime2_eq}, we obtain (a.s., as $T\to\infty$)
\begin{eqnarray*}
		\int_1^T h'(t)^2 dt &=&  \int_0^{\log T}  h'(e^\tau)^2   e^\tau d \tau
		= \int_0^{\log T} \left( z^\prime(\tau) +\frac{z(\tau)}{2}  \right)^2 d\tau  
\\
		&=& \int_{0}^{\log T} \left( b(Z(\tau)) + \frac {Z(\tau)}2\right) ^2 d\tau 
\\
		&\sim& {\log T} \int_0^{\infty} \left( b(x) + \frac x2\right) ^2 p(s) dx
\\
		&=& {\log T} \int_0^{\infty} \left( \left(  \frac {\log p}{2}\right)^\prime(x)  + \frac x2\right) ^2 p(x) dx
\\
		&=& {\log T} \int_0^{\infty}  \left( \frac {p^\prime(x)^2}{4p(x)} 
		+ \frac{x p^\prime(x)}{2} + \frac{x^2 p(x)}{4} \right) dx
\\
		&=& {\log T} \left( -\frac 12 + \int_0^{\infty}  \left( \frac {p^\prime(x)^2}{4p(x)} 
		+ \frac{x^2 p(x)}{4} \right) dx \right) 
\\
		&:=& -\frac {\log T}2 + \frac {\log T}{4} J(p).
\end{eqnarray*}
Taking into account condition \eqref{noexit}, it remains to solve the variational problem 
	\begin{equation*}
		\min \left\lbrace J(p) \Big|  \int_0^\infty p(x) dx = 1, p(0)=0 \right\rbrace 
\end{equation*}
over the class of densities concentrated on $[0, \infty)$.
After the variable change 
\[
	y(x) := p(x)^{1/2}, 
\]
the variational problem transforms into
\[
	\min\left\{  \int_0^\infty\left( 4y'(x)^2+x^2 y(x)^2 \right) dx \  \Big| \int_0^\infty y(x)^2 dx = 1, y(0)=0 \right\}.
\]	
We show in the next section that this minimum equals $6$; it is attained at the function
\[
	y(x) = (2/\pi)^{1/4}\, x\, \exp(-x^2/4).
\]
It follows that the asymptotic energy behavior for the optimal strategy is
\[
	\int_1^T h'(t)^2 dt \sim  \log T, \quad T\to\infty,
\]
i.e. the optimal choice of the shift in the adaptive setting leads to two times larger energy consumption than	for the optimal
strategy in the non-adaptive setting. 
	
In order to find the optimal shift, write
\[
	p(x)=y(x)^2=  (2/\pi)^{1/2}\, x^2\, \exp(-x^2/2)
\]
and we find from \eqref{B} -- \eqref{p}
\[
	b(x)= \frac12\, (\ln p)'(x) = \frac 1x - \frac{x}{2}. 
\]
Note that the density $p$ indeed satisfies the necessary condition \eqref{noexit}. 
	
Returning to the initial problem, we obtain the shift $\widetilde{b}(x)=\tfrac{1}{x}$, thus the strategy \eqref{hstrategy_eq} takes the form
\[
	h'(t) =  \frac{1}{h(t)-W(t)} \,. 
\]
Curiously, the optimal diffusion strategy is not only space-homogeneous but also time-homogeneous, unlike arbitrary strategies of this class.

\section{Solution of the variational problem} 
\label{s:6}

\subsection{Quantum harmonic oscillator}
Consider the Sturm--Liouville problem on the eigenvalues of a differential operator
\[
	\begin{cases}
		-4y''(x)+x^2 y(x)=\gamma \, y(x), & x\ge 0, \\
		y(0)=0.
	\end{cases}
\]
It represents a special case of the quantum harmonic oscillator equation, extensively studied by physicists, see \cite[\S 23]{LaLi}. 
Its solution is well known. Usually one considers this equation on the entire real line. When performing the restriction to  $[0,\infty)$,
one should take into account the boundary condition $y(0)=0$, hence, to keep the restrictions to  $[0,\infty)$ of {\it odd} solutions 
on $\R$ and multiply them by $\sqrt{2}$ in order to keep the normalization. We arrive at the orthonormal base $L_2[0,\infty)$ that consists
of the functions $\psi_k, k\in 2\N-1$, given by
\[
	\psi_k(x) = (2^k k!)^{-1/2} (2/\pi)^{1/4} H_k(x/\sqrt{2}) \exp(-x^2/4),
\]
where  $H_k(x)= (-1)^k e^{x^2}\tfrac{d^k}{dx^k}(e^{-x^2})$ are Hermite polynomials; these functions satisfy
\[
	-4\psi_k''(x)+x^2 \psi_k(x)=\gamma_k \psi_k(x),
\]
where $\gamma_k=2(2k+1)$. 
	
In particular, the minimal eigenvalue is $\gamma_1=6$, $H_1(x)=2x$, while the corresponding eigenfunction is
$\psi_1(x) = (2/\pi)^{1/4}\, x\, \exp(-x^2/4)$.
	
\subsection{Minimization}
Consider quadratic form 
\[
	G(y,z) := \int_0^\infty\left( 4y'(x)z'(x)+x^2 y(x)z(x) \right) dx.
\]
For twice differentiable functions satisfying additional assumption $y(0)=0$  integration by parts yields 
\[
	G(y,z) := \int_0^\infty\left( -4y''(x)+x^2 y(x)\right) z(x) dx.
\]
In particular, 
\[
	G(\psi_k,\psi_l) = \int_0^\infty \gamma_k \psi_k(x)\psi_l(x) dx =
	\begin{cases}
		\gamma_k,& k=l,\\
		0,& k\not= l,
	\end{cases}
\]
since $(\psi_k)$ is an orthonormal base.

If $y=\sum_{k\in 2\N-1} c_k \psi_k$, then
\[
	\int_0^\infty\left( 4y'(x)^2+x^2 y(x)^2 \right) dx 
	=G(y,y) = \sum_{k\in 2\N-1} c_k^2\gamma_k
	\ge  \sum_{k\in 2\N-1} c_k^2\gamma_1 = \gamma_1 \int_0^\infty y(x)^2 dx,
\]
and for $y=\psi_1$ the equality is attained in this chain.
Therefore,
\[
	\min\left\{  \int_0^\infty\left( 4y'(x)^2+x^2 y(x)^2 \right) dx \  \Big| \int_0^\infty y(x)^2 dx = 1 \right\} =\gamma_1 = 6.
\]	
\bigskip
	
	%% Acknowledgements
The authors are grateful to A.I.\,Nazarov for useful advice.

%%%%%%%%%%%%%%%%%%%%%%%%%%%%%%%%%%%%%%%%%%%%%%%%%%%%%%%%%%%%%%%%%%%%%%%%%%%%%%%%%%


\begin{thebibliography}{99}		
{\baselineskip=10pt \small}
		
\bibitem{BL} 
D.I.~Blinova, M.A.~Lifshits. Energy of taut strings accompanying  Wiener process and random walk
in a band of variable width. J. Math. Sci., 2022, 268, No. 5, 573--588.
%% DOI 10.1007/s10958-022-06228-6
		
\bibitem{Bor} 
A.N.~Borodin. Stochastic processes. Birkh\"auser, 2017.
%%	А.Н. Бородин, Случайные процессы. Лань, Санкт-Петербург, 2014.
		
\bibitem{KL2} 
Z.A.~Kabluchko,  M.A.~Lifshits. Adaptive energy saving approximation for stationary processes. 
Izvestia: Mathematics, 2019, 83, no.\,5, 932--956.  
%%DOI: https://doi.org/10.4213/im8840;            
%%DOI:10.1070/IM8840. (English ed.)
		
\bibitem{LaLi} 
L.D.~Landau, E.M.~Lifshitz.
Quantum Mechanics: Non-Relativistic Theory. Course of Theoretical Physics, Vol.3. 
3rd Edition. Elsevier, 2013.
%%		Л.Д.~Ландау, Е.М.~Лифшиц. Курс теоретической физики, т.3.
%%		Квантовая механика. Нерелятивистская теория. 6-е изд. М.: Физматлит, 2004.

\bibitem{BS} 
A.N.~Borodin, P.~Salminen. Handbook of Brownian motion. Facts and Formulae. Birk\-h\"auser, Basel, 1996. 
		
\bibitem{concMaj}
P.~Groeneboom. The concave majorant of Brownian motion. Ann. Probab., 1983, 11, no.4, 1016–-1027.
		
\bibitem{IKL} 
I.~Ibragimov, Z.~Kabluchko, M.~Lifshits. Some extensions of linear approximation and prediction problems 
for stationary processes. Stoch. Proc. Appl., 2019, 129, 2758--2782.
		
\bibitem{KL1} 
Z.~Kabluchko,  M.~Lifshits. Least energy approximations for processes with stationary increments. 
J. Theor. Probab., 2017, 30, no. 1, 268--296.
		
\bibitem{LS} 
M.~Lifshits, E.~Setterqvist. Energy of taut string accompanying Wiener process. Stoch. Proc. Appl., 2015, 125, 
401--427.
		
\bibitem{LSiu} 
M.A.~Lifshits, A.A.~Siuniaev. Energy of taut strings accompanying random walk. Probab. Math. Stat., 2021, 41, no.1, 9--23.
%%  https://doi.org/10.37190/0208-4147.41.1.2
		
\bibitem{Scher} 
E.~Schertzer. Renewal structure of the Brownian taut string. Stoch. Proc. Appl., 2018, 128, 487--504.
\end{thebibliography}
\end{document}